\documentclass[11pt]{amsart}
\calclayout
\usepackage{titlesec} 
\usepackage{amsfonts}
\usepackage{amsthm}
\usepackage{amssymb}
\usepackage{amsmath}
\usepackage{theoremref}
\usepackage{enumerate}
\usepackage{bbm}
\usepackage{mathtools}

\DeclarePairedDelimiter\abs{\lvert}{\rvert}%
\DeclarePairedDelimiter\norm{\lVert}{\rVert}%

\makeatletter
\let\oldabs\abs
\def\abs{\@ifstar{\oldabs}{\oldabs*}}
\let\oldnorm\norm
\def\norm{\@ifstar{\oldnorm}{\oldnorm*}}
\makeatother

\makeatletter
\g@addto@macro\bfseries{\boldmath}
\makeatother

\newcommand{\A}{\mathcal{A}}
\newcommand{\C}{\mathcal{C}}
\newcommand{\M}{\mathcal{M}}
\newcommand{\N}{\mathcal{N}}
\newcommand{\T}{\mathbb{T}}

\newcommand{\conj}[1]{\overline{#1}}
\newcommand{\D}{\mathbb{D}}

\newcommand{\Po}{\mathcal{P}}
\newcommand{\cD}{\conj{\mathbb{D}}}

\newcommand{\m}{\textit{m}}
 
\newcommand{\hd}{Hol(\D)}
\newcommand{\Hb}{\mathcal{H}(b)}
\newcommand{\hb}{\mathcal{H}(b)}

\newcommand{\h}{\mathcal{H}}

\renewcommand{\k}{\mathcal{K}}
\newcommand{\K}{\mathcal{K}}

\newtheorem{thm}{Theorem}[section]
\newtheorem{lemma}[thm]{Lemma}
\newtheorem{cor}[thm]{Corollary}
\newtheorem{prop}[thm]{Proposition}
\theoremstyle{definition}

\theoremstyle{definition}

\titleformat{\section}
{\normalfont\scshape\centering}{\thesection}{1em}{}

\titleformat{\subsection}
[runin]{\bfseries\scshape}{\thesubsection}{0.5em}{}

\newcommand{\Addresses}{{% additional braces for segregating \footnotesize
		\bigskip
		\footnotesize
		
		Adem Limani, \\ \textsc{Centre for Mathematical Sciences, Lund University, \\
			Lund, Sweden}\\
		\texttt{adem.limani@math.lu.se}
		
		\medskip
		
		Bartosz Malman, \\ \textsc{KTH Royal Institute of Technology, \\
			Stockholm, Sweden}\\
			\texttt{malman@kth.se}
		
	}}

\begin{document}
\title[Smooth approximations in de Branges-Rovnyak spaces]{\textbf{On the problem of smooth approximations in de Branges-Rovnyak spaces and connections to subnormal operators}}
\author{Adem Limani and Bartosz Malman}

\maketitle

\begin{abstract}
For the class of de Branges-Rovnyak spaces $\hb$ of the unit disk $\D$ defined by extreme points $b$ of the unit ball of $H^\infty$, we study the problem of approximation of a general function in $\hb$ by a function with an extension to the unit circle $\T$ of some degree of smoothness, for instance satisfying H\"older estimates or being differentiable. We will exhibit connections between this question and the theory of subnormal operators and, in particular, we will tie the possibility of smooth approximations to properties of invariant subspaces of a certain subnormal operator. This leads us to several computable conditions on $b$ which are necessary for such approximations to be possible. For a large class of extreme points $b$ we use our result to obtain explicit necessary and sufficient conditions on the symbol $b$ which guarantee the density of functions with differentiable boundary values in the space $\hb$. These conditions include an interplay between the modulus of $b$ on $\T$ and the spectrum of its inner factor.
\end{abstract}

\section{Introduction and the main result}

Let $b: \D \to \D$ be an analytic function, where $\D$ denotes the unit disk of the complex plane $\mathbb{C}$. The space $\hb$, introduced in \cite{debranges2015square} by de Branges and Rovnyak is the Hilbert space of analytic functions on $\D$ with reproducing kernels of the form \begin{equation} \label{kernelform} k_b(\lambda,z) = \frac{1-\conj{b(\lambda)}b(z)}{1-\conj{\lambda}z}, \quad \lambda,z \in \D. \end{equation}
De Branges and Rovnyak used $\hb$ and vector-valued generalizations of these spaces, together with the classical backward shift operator, as models for contractive operators on a Hilbert space. The general theory of $\hb$ was further developed by Sarason in a series of articles, and he summarized the work of himself and other researchers in the book \cite{sarasonbook} (see also the more recent monograph \cite{hbspaces1fricainmashreghi}, \cite{hbspaces2fricainmashreghi}). Importantly, the work of Sarason exhibited several connections between the theory of $\hb$-spaces and other parts of operator and function theory. 

The purpose of this note is to bring to light another such connection.  We will show how the problem of smooth approximations in $\hb$-spaces is connected to the theory of subnormal operators and derive some concrete results from this connection. The approximation problem can be stated in the following way. Let $X$ be some class of analytic functions in $\D$ with extensions to the unit circle $\T$ of some degree of smoothness. What are the conditions on $b$ which guarantee that for each $f \in \hb$ and each $\epsilon > 0$ we can find a function $g \in X \cap \hb$ such that $\|f-g\|_{\hb} < \epsilon$? The problem is only interesting in the case of $b$ being an extreme point of the unit ball, or equivalently, in the case of the boundary values of $b$ being so close to 1 in modulus that $\log( 1-|b|)$ is not integrable on $\T$. If this quantity is integrable then it is well-known, and it was proved by Sarason, that the space $\hb$ contains all polynomials and they form a dense subset of the space (see \cite{sarasonbook}). In contrast, in the extreme case only very few polynomials can be a member of $\hb$, and indeed it is often a non-trivial task to verify the existence of any non-zero function with any degree of boundary regularity in such a space. In this note, we will thus deal exclusively with the extreme case.

Function classes $X$ that we will consider in this note are to be more than simply continuous on $\T$. The continuous case has been previously investigated in \cite{comptesrenduscont}, and the conclusion is that no assumptions on $b$ whatsoever are necessary for this kind of approximation: the analytic functions extending continuously to the boundary (we will denote this space by $\A$) are always dense in $\hb$. In the present work, we want to instead study approximations by functions satisfying, for instance, classical H\"older or Lipschitz estimates, but also smaller classes such as $\A^k$, which is the space of analytic functions in $\D$ such that $f$, together with its first $k$ derivatives, extends continuously to $\T$. A simple and useful way to capture some notion of boundary smoothness of an analytic function in $\D$ is to specify the rate at which the Taylor coefficients of the function decrease. For a positive number $\alpha > 0$, $X_\alpha$ will thus denote the following Hilbert space of analytic functions in $\D$: 
\begin{equation}\label{Xalphadef}
X_\alpha = \{ f \in \hd : \| f\|^2_\alpha := \sum_{n=0}^\infty {(n+1)}^\alpha |f_n|^2 < \infty \},
\end{equation} the symbol $f_n$ denoting the n:th Taylor coefficient of the Taylor expansion of $f$ at the origin. The class $X_1$ is the classical Dirichlet space, and for large $\alpha$ the space $X_\alpha$ consists entirely of functions with a number of derivatives which extend continuously to $\T$. In fact, when $\alpha = 2k$ with positive integer $k$, then we basically recover the classical Hardy-Sobolev spaces. 

A subnormal operator $M$ is the restriction to an invariant subspace $\K$ of a normal operator on a Hilbert space $\h$. Thus $\K \subseteq \h$ and $M\K \subseteq \K$. A typical example is $\Po^2(\mu) \subseteq L^2(\mu)$ and the operator $M_z: f(z) \mapsto zf(z)$, where $\mu$ is a compactly supported measure in the complex plane, $L^2(\mu)$ is the usual space of $\mu$-square-integrable functions, and $\Po^2(\mu)$ is the closed subspace spanned by analytic polynomials. Any cyclic subnormal operator appears in this way for some measure $\mu$ (see, for instance, \cite{conway1991theory}). The measures $\mu$ which will appear in the context of our approximation problems are of the form $d\mu = w dA + \omega d\m$, where $dA$ and $d\m$ are (normalized) area and arclength measures on the disk $\D$ and circle $\T$, respectively, and $w$ and $\omega$ are integrable functions, $w$ being radial: $w(z) = w(|z|)$. In this case, the measure is supported on the closed unit disk and the operator $M_z$ is clearly a contraction. 

The following problem for measures of the above form has already been extensively studied in multiple contexts, sometimes in disguise: when is $M_z: \Po^2(\mu) \to \Po^2(\mu)$ completely non-unitary? By this we mean that there exists no $M_z$-invariant subspace of $\Po^2(\mu)$ on which $M_z$ acts as a unitary operator. In the setting which we consider below, we will need the slightly stronger assumption that $M_z$ is \textit{completely non-isometric}, that is, admits no $M_z$-invariant subspace on which this operator acts as an isometry. More explicitly, there should exist no non-zero element $f \in \Po^2(\mu)$ which satisfies \[ \|f\|^2_{\mu,2} := \int_{\cD} |f(z)|^2 d\mu(z) = \int_{\cD} |z^nf(z)|^2 d\mu(z) = \|M_z^nf\|_{\mu,2}, \] for all positive integers $n$. The number of works on this problem is huge, and we would like to in particular mention articles \cite{kriete1990mean} and \cite{khrushchev1978problem} which will provide interesting examples relevant for the present work. If $w$ is, say, continuous and positive on $\D$ (as it will be in our applications), then convergence of a sequence of polynomials $\{p_n\}_{n \geq 1}$ in $\Po^2(\mu)$ to $f$ implies that the restriction of $f$ to $\D$ is an analytic function on $\D$, and it is easy to see that $M_z: \Po^2(\mu) \to \Po^2(\mu)$ is completely non-isometric if and only $\Po^2(\mu)$ does not contain any non-zero element $f \in L^2(\mu)$ which vanishes on $\D$, and thus lives only on the circle $\T$. Note that in this case, and this is very important for our purposes, $\Po^2(\mu)$ can be interpreted as a space of analytic functions on $\D$. However, even in the case of measures of the form above, a complete characterization of when this happens seems to be very difficult to obtain. We will see how this problem is connected to smooth approximations in $\hb$, and that the many available partial results will nevertheless help us reach several interesting conclusions.

Our main theorem will be stated in terms of a measure of the form mentioned above. Given an extreme point $b$ of the unit ball of $H^\infty$, and a parameter $\alpha > 0$, we shall now construct a measure $\mu = \mu(b,\alpha)$ on the closed unit disk $\cD$ and a corresponding space $\Po^2(\mu)$, which encodes in properties of the operator $M_z: \Po^2(\mu) \to \Po^2(\mu)$ the possibility of approximation of an arbitrary function in $\hb$ by a function in $X_\alpha$. Let \begin{equation} \label{diskweight} w(z) = w_{\alpha-1}(z) := (1-|z|^2)^{\alpha-1} \end{equation} be the weight on the disk $\D$,\begin{equation}\label{circleweight} \Delta(z) := \sqrt{1-|b(z)|^2} \end{equation} be the weight on $\T$, and define the measure \begin{equation}
\label{measure} d\mu = d\mu(b,\alpha) := w_{\alpha-1} dA + \Delta^2 d\m.
\end{equation}

The following is our main theorem on the connection between the smooth approximation problem in $\hb$ spaces and the structure of invariant subspaces of the operator $M_z: \Po^2(\mu) \to \Po^2(\mu)$. In the statement below, the space $\N^+$ is the Smirnov class of the unit disk. Recall that a function $f$ belongs to $\N^+$ if and only if it can be expressed as $f = u/v$, where $u,v$ are bounded and $v$ is an outer function (the elementary theory of the Smirnov class can be found in \cite{garnett}). If $f$ is a function in $\Po^2(\mu)$, we denote by $[f]$ the smallest closed $M_z$-invariant subspace containing $f$. Clearly, $g \in [f]$ if and only if there exists a sequence of polynomials $\{p_n\}_{n \geq 1}$ such that $\lim_{n \to \infty} \|fp_n - g\|_{\mu,2} = 0$. 

\begin{thm} \thlabel{maintheorem}
Let $b$ be an extreme point of the unit ball of $H^\infty$, $\alpha > 0$ and $\mu$ be defined by \eqref{measure}. Further, let $\theta$ be the inner factor of $b$ and the space $X_\alpha$ be defined according to \eqref{Xalphadef}. The following statements are equivalent.

\begin{enumerate}[(i)]
\item The linear manifold $X_\alpha \cap \hb$ is dense in $\hb$.
\item The operator $M_z: \Po^2(\mu) \to \Po^2(\mu)$ is completely non-isometric and the following property holds: if $f \in \N^+$ is a function contained in the invariant subspace $[\theta] \subset \Po^2(\mu)$, then $f/\theta \in \N^+$. 
\end{enumerate}
\end{thm}

The density of function classes $X_\alpha$ in $\hb$ is thus reduced to the study of some specific structure of certain invariant subspaces of $\Po^2(\mu)$. The inner function division property described in $(ii)$ should be compared to the arguments appearing in \cite{comptesrenduscont}, where a similar condition plays a role in establishing the density of the disk algebra $\A$ in $\hb$. Moreover, it can be deduced from results of \cite{smoothdensektheta} that the density of smooth functions $\A^\infty := \cap_{n \geq 1} \A^k$ in a model space $K_\theta$ (that is, an $\hb$ space generated by an inner function $b = \theta$) is equivalent to this division property with $\mu = dA$.

The rest of the paper is structured as follows. In the next section, we present applications of the \thref{maintheorem} which will provide us a good understanding of what the different properties of the symbol $b$ imply for our approximation question. In particular, we shall illustrate three principal ways in which smooth approximations are prohibited. We will also construct a very large family of symbols $b$ which avoid in some ways the three mentioned obstructions, and we will show that for this family of symbols $b$, the set $X_\alpha \cap \hb$ is dense in $\hb$. In Section \ref{Xalphasection}, we present some basic facts about the $X_\alpha$ spaces. The proof of \thref{maintheorem} is deferred to Sections  \ref{isomsection}, \ref{divsec} and \ref{suffsec}.

%We want to remark that the connection between $X_\alpha$ and $w_{\alpha-1}$ is the norm equivalence $\|f\|^2_\alpha \simeq \int_\D |f|^2 w_{\alpha -1} dA$, and these two objects can be replaced in the above theorem by any space $X$ with norm of the type $\|f\|^2_X = \sum_{n=1}^n c_n |f_n|^2$ and a weight $w$ on $\D$ for which the same norm equivalence holds. Our results can in this way be generalized. 

\section{Applications of the theorem}

\thref{maintheorem} and the theory of subnormal operators allows for immediate deduction of interesting facts about smooth approximations on $\hb$-spaces, defined by extreme points $b$ of the unit ball of $H^\infty$. In fact, in parallel with this work, the authors developed the theory in \cite{ptmuinnner} partly for the purpose of this application, and that work will provide us with very general results and interesting examples. Outside this reference, also examples from \cite{khrushchev1978problem} and \cite{kriete1043splitting} will be very useful.

We start by presenting the explicit quantitative conditions that we have found which prohibit smooth approximations in $\hb$-spaces. As mentioned earlier, our investigation exhibits three principal ways in which this happens. The first way deals with the structure of the carrier set of the weight $\Delta$ defined in \eqref{circleweight} above, second one deals with the size of the weight $\Delta$ on the carrier set, and the third deals with an interplay between the singularities of the inner factor of $b$ and the location of the mass of $\Delta$. For the purpose of our presentation we will need to introduce classes of sets and measures on the circle $\T$. For a closed set $E$ on $\T$, we consider the complement $\T \setminus E = \cup_{k} A_k$, where $A_k$ are disjoint open circular arcs. The set $E$ is a \textit{Beurling-Carleson set} if the following condition is satisfied: 

\begin{gather} \label{carlesoncond}
%\int_{\T \setminus E} \log \left( \frac{1}{\text{dist}(\zeta, E) }\right) dm(\zeta) \sim 
\sum_k |A_k|\log \left( \frac{1}{ |A_k| } \right) < \infty. 
\end{gather}
These sets appear in the theory of Bergman spaces of Korenblum (see \cite{hedenmalmbergmanspaces}), often with the extra condition that the Lebesgue measure $|E|$ of $E$ is zero, and they appear also in the theory of subnormal operators, where they usually have positive measure (see \cite{kriete1990mean} or \cite{khrushchev1978problem}). We will use Beurling-Carleson sets of both zero and positive measure in the presentation of our examples. In the examples below, we let $\mu$ be given by \eqref{measure} and denote by \begin{equation} \label{setEdef} E := \{ \zeta \in \T : |b| < 1 \} = \{ \zeta \in \T : \Delta > 0 \} \end{equation} the carrier set of $\Delta$. Our first application of \thref{maintheorem} shows that the structure of $E$ can prohibit smooth approximations.

\begin{cor}
\thlabel{hruscevexample} Assume that $E$ is such that there exists no Beurling-Carleson set $G$ of positive measure such that $|G \setminus E| = 0$. Then $X_\alpha \cap \hb$ is not dense in $\hb$, for any $\alpha > 0$. 
\end{cor}

The assumption means that $E$ essentially does not contain Beurling-Carleson sets of positive measure. In particular, $E$ contains no interval, but of course much more is true. The corollary follows from applying \cite[Theorem 10.1]{khrushchev1978problem} to our criterion given in \thref{maintheorem}. See also \cite[Theorem 1.2]{kriete1990mean}. The cited results imply that $L^2(\mu|\T)$ is a summand of $\Po^2(\mu)$, and on this subspace $M_z$ certainly acts as an isometric operator. Moreover, sets of Cantor-type are constructed in \cite{khrushchev1978problem} which satisfy the condition in \thref{hruscevexample} and can be used to produce our first explicit examples of $\hb$ in which no smooth approximations are possible. Take such a set $E$ and any $b$ such that $|b| = 1$ on $\T \setminus E$ but $|b| < 1$ almost everywhere on $E$. Irrespective of the inner factor of $b$, the linear manifolds $X_\alpha \cap \hb$ are never dense in $\hb$ for such choices of $b$. In fact, it can be deduced from the proof of \thref{isomnodense} below (see the remark below the proof) that if $\theta$ is the inner factor of $b$, then such a space $\hb$ will contain a non-zero function in the class $X_\alpha$ if and only if $K_\theta$ will contain one, since $K_\theta \cap X_\alpha = \hb \cap X_\alpha$ in that case.

The previous example shows how the structure of the carrier of $\Delta$ can disallow smooth approximations. The size of $\Delta$ also plays a role, which is measured by (local) integrability of $\log \Delta$. More precisely, Volberg constructed a very small weight $W$ on $\T$ with the property that $\log W$ is not integrable on any arc of $\T$, and showed that $M_z$ on $\Po^2(\mu)$ with $d\mu = dA + W d\m$ contains $L^2(W d\m)$ as a direct summand. This example is contained in the editor's commentary to \cite{kriete1043splitting}. Meanwhile, it is worth mentioning that if $\log W$ is integrable on some arc of $\T$, then $L^2(W d\m)$ is not a direct summand of $\Po^2(\mu)$ (see Theorem D, \cite{kriete1043splitting}). By re-scaling Volberg's weight $W$ so that $W < 1/2$ almost everywhere on $\T$, we can then construct an outer function $b$ with modulus satisfying $|b|^2 = 1-W^2$, and then \thref{maintheorem} implies the following.

\begin{cor} \thlabel{volbergexample} There exists an outer function $b$ with $|b| < 1$ almost everywhere on $\T$, thus $E = \T$ in \eqref{setEdef}, such that $X_\alpha \cap \hb$ is not dense in $\hb$, for any $\alpha > 0$. 
\end{cor}

The proof of \thref{isomnodense} below will show that if the outer function $b$ is constructed as explained above, then in fact we will have $X_\alpha \cap \hb = \{0 \}$ for any $\alpha > 0$.

The third way in which smooth approximations can be prohibited is by existence of certain inner factors of $b$, and to clarify this we will need to introduce a decomposition. It is easy to see that every singular measure $\nu$ on $\T$ can be uniquely decomposed into mutually singular parts, one concentrated on a countable union of Beurling-Carleson sets of measure zero and one part which vanishes on such sets. More explicitly, \begin{equation}
\label{BCKdecomp} \nu = \nu_\C + \nu_\K \end{equation} and there exists a set $E = \cup_k E_k$ with each $E_k$ being Beurling-Carleson with $|E_k| = 0$ such that $\nu_\C(B) := \nu_\C(B \cap E)$, and $\nu_\K(F) = 0$ for each Beurling-Carleson set $F$ of Lebesgue measure zero. The decomposition in \eqref{BCKdecomp} corresponds to a factorization $S_\nu = S_{\nu_\C}S_{\nu_\K}$. In the class of model spaces, it is the factor $S_{\nu_\K}$ that prohibits smooth approximations (see \cite{smoothdensektheta}). The situation is different for general $\hb$-spaces, and we will see below that it is the location of the mass of $\nu_\K$ that plays a decisive role. 

\begin{cor} \thlabel{Cyclicexample} Let $F$ be the (closed) support of $\Delta$. If $S_{\nu}$ is the inner factor of $b$ and $\nu_\k(\T \setminus F) > 0$, then $X_\alpha \cap \hb$ is not dense in $\hb$, for any $\alpha > 0$. 
\end{cor}

This result follows from cyclicity results of \cite{ptmuinnner}. The condition $\nu_\k(\T \setminus F) > 0$ implies that $S_\nu$ contains an inner factor which is a cyclic vector for $M_z$ on $\Po^2(\mu)$, and thus condition (ii) of \thref{maintheorem} on $[S_\nu]$ cannot be satisfied. Note that the set $F$ is the \textit{support} of $\Delta$, which can indeed be larger than the carrier set $E$. If $E$ is a Beurling-Carleson set, then $F = E$. 

We now turn to positive results. We shall construct a very large class of symbols $b$ for which $X_\alpha \cap \hb$ is dense in $\hb$, for any $\alpha > 0$. In particular, the class will contain all $b$ which extend continuously to the boundary and the inner factor $S_\nu$ of $b$ is such that the mass of $\nu_\K$ is placed appropriately and accordingly with the restriction given in \thref{Cyclicexample}. The class is however much larger and the recipe is as follows. 

\begin{enumerate}[(i)]
\item Let $\{E_k\}_k$ be a sequence of Beurling-Carleson subsets of $\T$ of positive measure, $E = \cup_k E_k$ and $b_0$ be any outer function bounded by $1$, satisfying $|b_0| = 1$ on $\T \setminus E$ and \[ \int_{E_k} \log(1-|b_0|^2) d\m > -\infty\] for each set $E_k$,

\item $\nu_\K$ be a singular measure which vanishes on Beurling-Carleson sets of measure zero and for which we have $\nu_\K(\T \setminus E) = 0$,

\item $\nu_\C$ be an arbitrary singular measure concentrated on a countable union of Beurling-Carleson sets of measure zero,

\item $B$ be an arbitrary convergent Blaschke product.
\end{enumerate}

Note how the first two conditions in different ways ensure that the possible "bad" behaviour of $b$ which appears in \thref{hruscevexample}, \thref{volbergexample} and \thref{Cyclicexample} is avoided.  

\begin{cor} \thlabel{mainapproxtheorem}
Let $b = BS_{\nu_\C + \nu_K}b_0$ be constructed according to the above specifications. Then for each $\alpha > 0$, the set $X_\alpha \cap \hb$ is dense in $\hb$. 
\end{cor}

This result follows from \cite[Theorem 1.2]{ptmuinnner}, where condition (ii) of \thref{maintheorem} is verified for the associated measures $\mu(b,\alpha)$. Note the similarity with the case of non-extreme $b$. In that case $\log(1-|b|^2)$ is integrable on the entire circle $\T$, and then smooth approximations are certainly possible, since then even the analytic polynomials are dense in the space. If the weight $(1-|b|^2)$ lives only on a part of the circle, then $b$ is an extreme point and thus the set of polynomials contained in $\hb$ forms only a finite dimensional subspace. However, if the subset of $\T$ on which the weight $(1-|b|^2)$ lives satisfies our structural assumption, and is $\log$-integrable there, then approximations by functions from smoothness classes $X_\alpha$ are possible (at least when $b$ is outer), similarly to the non-extreme case. 

Before ending this section, and proceeding with the proof of \thref{maintheorem}, we would like to comment on some problems which we did not overcome. One can see that in \thref{mainapproxtheorem}, we have not been able to fully extend to approximations by smooth functions in the class $\A^\infty = \cap_{\alpha > 0} X_\alpha$, that is, by functions analytic in $\D$ with boundary values in $C^\infty$. We believe this to be only a deficiency of our particular method and such approximations should be possible for the class of $\hb$ spaces exhibited in \thref{mainapproxtheorem}. In fact, we believe further development of the general theory of $\Po^2(\mu)$-spaces  for measures $\mu$ of the kind given in \eqref{measure} will allow us to fill in the gap. Given our examples, another natural question is if the parameter $\alpha > 0$ plays a significant role, or if density of $X_\alpha \cap \hb$ in $\hb$ for small $\alpha$ is equivalent to density for all $\alpha$, just as it is in $K_\theta$ (see \cite{smoothdensektheta}).

A third question for which we have not found a satisfactory answer is: what happens if the weight $(1-|b|^2)$ is $\log$-integrable on the carrier set $E$, but this set is not (up to sets of measure zero) a union of a sequence of Beurling-Carleson sets? In this case we can perform the following decomposition of the weight. Consider the quantity \[ \gamma = \sup_F |F| \] where $F$ is a Beurling-Carleson subset of $E$. If $\{F_n\}_n$ is a sequence of Beurling-Carleson subsets of $E$ for which we have $\gamma = \lim_n |F_n|$, then the set $E \setminus (\cup_n F_n)$ has non-zero measure (since else $E\setminus N = \cup_n F_n$ for a set $N$ of measure zero) and certainly satisfies the property appearing in \thref{hruscevexample}. The example in \cite{khrushchev1978problem}, on which \thref{hruscevexample} is based, hints that the operator $M_z$ on $\Po^2(\mu)$ could in this case have a non-trivial isometric part, but we have not been able to settle this.

\section{$X_\alpha$-spaces and their duality}
\label{Xalphasection}

In the following sections we prove \thref{maintheorem}, and we start by recalling some simple properties of the spaces $X_\alpha$.

Recall, for $\alpha > 0$, the definition of the space $X_\alpha$ in \eqref{Xalphadef}. Naturally, we could also allow the parameter $\alpha$ to take on any real value, instead of requiring it be positive as in the earlier sections. We will need this broader range of parameters, and we will now describe a simple but crucial duality relation and a norm equivalence. 

The spaces $X_\alpha$ are Hilbert spaces, but instead of seeing $X_\alpha$ as dual of itself with respect to the inner product, we shall instead find use of the \textit{Cauchy duality} between $X_\alpha$ and $X_{-\alpha}$. The Cauchy duality is realized by \[\lim_{r \to 1} \int_\T f_r\conj{g_r}dm, \quad f \in X_\alpha, g \in X_{-\alpha},\] where $f_r(z) = f(rz)$ and $g_r(z) = g(rz)$ are the usual dilations of $f$ and $g$. Indeed, if $\{f_n\}_{n \geq 0}$ and $\{g_n\}_{n \geq 0}$ are the sequences of Taylor coefficients of $f$ and $g$, then we easily see that for the above limit we have \begin{gather*} \lim_{r \to 1} \sum_{n = 0}^\infty f_n\conj{g_n}r^{2n} = \lim_{r \to 1} \sum_{n = 0}^\infty (n+1)^{\alpha/2} f_n \frac{\conj{g_n}}{(n+1)^{\alpha/2}}r^{2n} \\
= \sum_{n = 0}^\infty (n+1)^{\alpha/2} f_n \frac{\conj{g_n}}{(n+1)^{\alpha/2}} 
\end{gather*} where passage to the limit $r = 1$ is justified by the fact that $\{(n+1)^{\alpha/2} f_n\}_n$ and $\{(n+1)^{-\alpha/2} g_n\}_n$ are square-summable sequences by definition of $X_\alpha$ and $X_{-\alpha}$. Under this pairing $X_{-\alpha}$ is isometrically isomorphic to the dual space of $X_\alpha$. Clearly the Cauchy dual of $X_{-\alpha}$ is $X_{\alpha}$, and thus the usual weak-star topology on $X_{-\alpha}$ coincides with the weak topology on $X_{-\alpha}$. 

Again restricting the parameter $\alpha$ to be positive, a computation shows that for $\alpha$ fixed, we have that \[\int_\D |z|^{2n} (1-|z|^2)^{\alpha -1} dA(z) \simeq \frac{1}{n^\alpha}.\] This implies immediately that the square of the norm on $X_{-\alpha}$ is equivalent to \[ \int_\D |f(z)|^2(1-|z|^2)^{\alpha-1} dA(z). \] 

\section{Isometric part of $M_z$ prohibits smooth approximation}
\label{isomsection}

In this section, we shall treat one of the implications in \thref{maintheorem}. Namely, if the symbol $b$ and $\alpha > 0$ is such that the corresponding $\Po^2(\mu)$-space contains a non-trivial subspace on which $M_z$ acts as an isometric operator, then $X_{\alpha} \cap \hb$ is not dense in $\hb$. One of our main tools for this will be the following structure theorem for $\hb$-spaces in the extreme case. The result appears in \cite{comptesrenduscont}, and is generalized to other classes of spaces in \cite{jfabackshift}. The symbol $P_+$ denotes the orthogonal projection from $L^2$ of the circle $\T$ onto the Hardy space $H^2$.

\begin{prop} \thlabel{normformula}
Let $b$ be an extreme point of the unit ball of $H^\infty$, $$E = \{ \zeta \in \T : |b(\zeta)| < 1 \},$$ and let $\Delta = \sqrt{1-|b|^2}$ be a function on the circle $\T$, defined in terms of boundary values of $b$ on $\T$. For $f\in \Hb$
the equation
 $$P_+ \conj{b}f = -P_+ \Delta g$$
has a unique solution $g\in L^2(E)$, and the map $J:\Hb\to H^2\oplus L^2(E)$  defined by $$Jf=(f,g),$$ 
is an isometry. Moreover,  $$J(\Hb)^\perp = \Big\{ (bh, \Delta h) : h \in H^2 \Big\}.$$ 
\end{prop}

We will also need the following classical result on the structure of $M_z$-invariant subspaces of $L^2(\T)$. For a proof, see for instance \cite{helsonbook}.

\begin{prop} \thlabel{beurling-wiener}
Let $M_z^*$ be the adjoint of the shift on $L^2(E)$. The $M^*_z$-invariant subspaces of $L^2(\T)$ are of the form \[L^2(F) = \{ f \in L^2(\T) : f = 0 \text{ almost everywhere on } \T \setminus F\}\] where $F$ is a measurable subset of $\T$, or of the form \[U \conj{H^2} = \{ U\conj{f} : f \in H^2\}\] where $U$ is a unimodular function.
\end{prop}

The following observation is crucial.

\begin{lemma} \thlabel{Jkbshiftdense}
If $J$ is the embedding in \thref{normformula}, then \[Jk_b(\lambda,z) = \Big( k_b(\lambda,z), \frac{-\Delta \conj{b(\lambda)} }{1-\conj{\lambda}z} \Big), \] and if $b(\lambda) \neq 0$, then the $M^*_z$-invariant subspace of $L^2(E)$ generated by $\frac{-\Delta \conj{b(\lambda)} }{1-\conj{\lambda}z}$ equals $L^2(E)$. 
\end{lemma}

\begin{proof}
A simple computation involving \thref{normformula} will reveal that the tuple $Jk_b(\lambda,z)$ has the indicated form. Note that $\frac{-\Delta \conj{b(\lambda)} }{1-\conj{\lambda}z}$ is non-zero almost everywhere on $E$, and moreover, this function is not $\log$-integrable. If a non-zero function is contained in $U\conj{H^2}$, then it is $\log$-integrable, and subspaces of the form $L^2(F)$ contain functions vanishing on sets of positive measure unless $F = E$. The claim follows from \thref{beurling-wiener}.
\end{proof}

We will use the above corollary in the proof of the following proposition, which is our first step towards proving \thref{maintheorem}.

\begin{prop} \thlabel{isomnodense}
Let $b$ be an extreme point of the unit ball of $H^\infty$, $\alpha > 0$ and $\mu$ be given by \eqref{measure}. If $M_z: \Po^2(\mu) \to \Po^2(\mu)$ admits a non-trivial invariant subspace on which $M_z$ acts as an isometry, then the set $X_\alpha \cap \hb$ is not dense in $\hb$.
\end{prop}

\begin{proof}
Since $\Po^2(\mu) \subseteq \Po^2(w_{\alpha-1}dA) \oplus L^2(\Delta^2 d\m)$, the assumption implies that there exists a tuple of the form $(0, t) \in \Po^2(\mu)$, where $t \in L^2(\Delta^2 d\m)$ is non-zero. Let $S$ be the $M_z$-invariant subspace of $L^2(\Delta^2 d\m)$ generated by $t$. Then $\{0 \} \oplus S$ is contained in $\Po^2(\mu)$. The operator $h \mapsto \Delta h$ is a unitary mapping between $L^2(\Delta^2 d\m)$ and $L^2(E)$, and this mapping commutes with the shift operators on the spaces. It follows that $\Delta S = \{ h \Delta : h \in S \}$ is an $M_z$-invariant subspace of $L^2(E)$. 

For each $h \in S$ there exists a sequence of polynomials $\{p_n\}_{n\geq 1}$ such that $(p_n, p_n) \to (0, h)$ in the norm of $\Po^2(w_{\alpha-1}dA) \oplus L^2(\Delta^2 d\m)$. By multiplying the first coordinate by $b$ and using the norm equivalence presented in Section \ref{Xalphasection}, we clearly are in the following situation:
\[ \|bp_n\|_{-\alpha} \to 0 \] and \[ \int_{E} |\Delta p_n - \Delta h|^2 d\m \to 0.\]
Assume that $s$ is a function in $X_\alpha \cap \hb$ and consider $Js = (s,g)$, where $J$ is the embedding of \thref{normformula}. By the description of the annihilator of $J\hb$ from same theorem, we have that 
\[ 0 = \int_\T s\conj{bp_n} d\m + \int_E g\Delta \conj{p_n} d\m. \] The convergence of $bp_n$ to $0$ is in the norm of the Cauchy dual of $X_\alpha$, and thus in the limit as $n \to \infty$ we obtain \[0 = \int_E g \Delta \conj{h} d\m. \] Since $h \in S$ was arbitrary, we conclude that $g$ is orthogonal to $\Delta S$ in $L^2(E)$. Since $\Delta S \neq \{0\}$, the second component $g$ of $Js$ is thus contained in the proper subspace $(\Delta S)^\perp$ of $L^2(E)$, whenever $s \in X_\alpha \cap \hb$. Note that $(\Delta S)^\perp$ is invariant for $M_z^*$. It follows now from \thref{Jkbshiftdense} that tuples of the form $Js$ cannot approximate tuples of the form $Jk_b(\lambda,z)$ whenever $b(\lambda) \neq 0$, since then the second component of $Jk_b(\lambda, z)$ does not lie in $(\Delta S)^\perp$.
\end{proof}

We remark that if the invariant subspace on which $M_z$ acts as an isometry in \thref{isomnodense} is the space $L^2(\mu|\T)$, then it follows from the proof that we have $Js = (s,0)$ for any $s \in X_\alpha$. It follows then from \thref{normformula} that $s \in K_\theta$, where $\theta$ is the inner factor of $b$. If $b$ is outer, then $X_\alpha \cap \hb = \{ 0\}$.

\section{Necessity of the division property}
\label{divsec}
Before going into the proof of the main result of this section, we make the following observation. By developments of Section \ref{isomsection}, in order to prove \thref{maintheorem}, we may assume below that $M_z$ is completely non-isometric on $\Po^2(\mu)$, and thus $\Po^2(\mu)$ is a space of analytic functions. For any $f \in H^2$, let $\{p_n\}$ be a sequence of polynomials which converges to $f$ in $H^2$ and almost everywhere on $\T$. It is clear that this sequence will also be a Cauchy sequence in our $\Po^2(\mu)$ space, and thus $f$ will be a member of $\Po^2(\mu)$. The part of $f$ which lives on $\T$ will be represented in $\Po^2(\mu)$ by the usual non-tangential boundary values of $f$. Moreover, any bounded analytic function is a multiplier of $\Po^2(\mu)$, and it acts in the natural way on the boundary. 

The following result, together with the previous one, will complete the proof of the implication $(i) \implies (ii)$ in \thref{maintheorem}. 

\begin{prop}
Let $b$ be an extreme point of the unit ball of $H^\infty$, $\alpha > 0$ and $\mu$ be given by \eqref{measure}. Further, let $\theta$ be the inner factor of $b$. Assume that $X_\alpha \cap \hb$ is dense in $\hb$. Then $M_z :\Po^2(\mu) \to \Po^2(\mu)$ is completely non-isometric, and moreover whenever $f \in \N^+$ is contained in $[\theta]$, then $f/\theta \in \N^+$. 
\end{prop}

\begin{proof}
By the above discussion, $M_z$ is completely non-isometric, $\Po^2(\mu)$ is a space of analytic functions, and it remains to show that the division property in $[\theta]$ holds. Assume therefore that $\lim_{n\to\infty} \|\theta p_n - f\|_{\mu,2}  = 0$, where $f = u/v  \in \N^+$, $u,v$ are bounded analytic functions with $v$ being outer, and $\{p_n\}_{n \geq 1}$ is a sequence of bounded analytic functions. By multiplying the convergent sequence by $v$ we can assume that $f$ is bounded.  Let $b_0$ be the outer factor of $b$. By further multiplication by $b_0$ we deduce from the above convergence that \[ \lim_{n\to \infty} \int_\D |bp_n - b_0f|^2 w_{\alpha -1}dA = 0, \] and we also have that\[ \lim_{n \to \infty} \int_E  |\Delta \theta p_n - \Delta f|^2 d\m = \lim_{n \to \infty} \int_E  |\Delta p_n -  \Delta \conj{\theta} f|^2 d\m= 0.\] Now, take any $s \in X_\alpha \cap \hb$. If $Js = (s,g)$, then the above two limits being zero can be used to show, just as in the proof of \thref{isomnodense}, that \[ \int_\T s\conj{b_0f} d\m + \int_E g \Delta \theta \conj{f} d\m = \lim_{n \to \infty} \left( \int_\T s\conj{b p_n} d\m + \int_E g \Delta \conj{p_n} d\m \right) = 0. \] The last equality follows by the orthogonality relation of \thref{normformula}. By the density of the elements of the form $Js$ in $J\hb$, this implies that $(b_0f,\Delta \conj{\theta} f ) \perp J\hb$, and so $(b_0f,\Delta \conj{\theta} f) = (bh, \Delta h)$ for some $h \in H^2$, again by \thref{normformula}. Thus $f/\theta = h \in \N^+$, and the proof is complete.
\end{proof}

\section{Sufficiency of the conditions on $\M_z$ for approximation}
\label{suffsec}
We prove now the implication $(ii) \implies (i)$ of \thref{maintheorem}.

\begin{prop} Let $b$ be an extreme point of the unit ball of $H^\infty$, $\alpha > 0$ and $\mu$ be given by \eqref{measure}. Further, let $\theta$ be the inner factor of $b$. Assume that $M_z: \Po^2(\mu) \to \Po^2(\mu)$ admits no non-trivial invariant subspace on which $M_z$ acts as an isometric operator, and moreover whenever $f \in \N^+$ is in $[\theta]$, then $f/\theta \in \N^+$. Then $X_\alpha \cap \hb$ is dense in $\hb$.
\end{prop}

\begin{proof}
Let us assume that $f \in \h(b)$ is orthogonal to $X_\alpha \cap \h(b)$. We will show that the assumptions in the statement of the proposition imply that $f = 0$. Because the mapping $J$ in \thref{normformula} is an isometry, it follows that $Jf$ is orthogonal to $J(X_\alpha \cap \h(b))$. Note that $J(X_\alpha \cap \h(b))$ is a subset of $X_\alpha \oplus L^2(E)$. 

We claim that, under the Cauchy duality between $X_\alpha$ and $X_{-\alpha}$, we have \begin{equation}
\label{preanneq} J(X_\alpha \cap \h(b)) = \cap_{h \in H^2} \ker l_h,
\end{equation} where $l_h$ is the functional on $X_\alpha \oplus L^2(E)$ given by Cauchy pairing with $(bh, \Delta h) \in X_{-\alpha} \oplus L^2(E)$, $h \in H^2$. Indeed, any pair $(c,d) \in X_\alpha \oplus L^2(E)$ annihilated by $\{l_h\}_{h \in H^2}$ lies in $J(\h(b))$ by \thref{normformula}. Conversely, again by same formula, any function in $J(X_\alpha \cap \h(b))$ will be annihilated by $\{l_h\}_{h \in H^2}$. 

Equip $X_{-\alpha} \oplus L^2(E)$ with the weak*-topology as the dual space of $X_{\alpha} \oplus L^2(E)$. The fact that $Jf$ annihilates $J(X_{\alpha} \cap \h(b))$ and \eqref{preanneq} holds implies that $Jf$ is contained in the weak*-closure of the linear manifold $\{l_h\}_{H^2} \subseteq X_{-\alpha} \oplus L^2(E)$. Since the dual of $X_{-\alpha}$ is $X_{\alpha}$ and vice versa, it follows that the weak*-topology coincides with the weak-topology on $X_{-\alpha} \oplus L^2(E)$. Thus $Jf$ is contained in the weak-closure of $\{l_h\}_{h \in H^2}$. Since $\{l_h\}_{h \in H^2}$ is a convex set, its weak-closure and norm-closure coincide. Thus there exists a sequence $\{h_k\}_{k \geq 1}$ with $h_k \in H^2$ such that \begin{equation}
\label{seqconv} (bh_k, \Delta h_k) \to Jf := (f,g)
\end{equation} in the norm of $X_{-\alpha} \oplus L^2(E)$. It follows from \eqref{seqconv} by multiplication of the second coordinate by $b$, and from norm equality presented in Section \ref{Xalphasection}, that $\{bh_k\}_k$ is a Cauchy sequence in $\Po^2(\mu)$, and thus converges to a function in $\Po^2(\mu)$, which of course must be $f$. We now use the assumption that $M_z$ is completely non-isometric: since $f$ is a function in $H^2$, the discussion in the beginning of Section \ref{divsec} implies that $bg = \Delta f$ and thus that $g = \Delta f/b$ on $E$. Moreover, $Jf = (f, \Delta f/b)$, and by our assumption $f/\theta \in \N^+$, thus actually $f/\theta \in H^2$ by the Smirnov maximum principle. Since we are assuming that $f \in \h(b)$, by \thref{normformula} we get that \begin{equation} \label{projzero}
0 = P_+(\conj{b}f + \Delta g) = P_+(\conj{b}f + \Delta^2 f/b) = P_+(|b|^2f/b + \Delta^2 f/b) = P_+(f/b).
\end{equation} In the above computation we showed that $\conj{b}f + \Delta g = f/b$ on $\T$, and thus $f/b$ has square-integrable boundary values. Since $f/\theta \in H^2$, it follows from the Smirnov maximum principle that $f/b \in H^2$. Thus $f/b$ is an analytic function which projects to $0$ under $P_+$, which implies that $f/b = 0$, and consequently $f=0$.
\end{proof}

Note that, together with developments of Section \ref{isomsection}, the proof of \thref{maintheorem} is now complete.

\bibliographystyle{siam}
\bibliography{mybib}

\begin{thebibliography}{10}

\bibitem{comptesrenduscont}
{\sc A.~Aleman and B.~Malman}, {\em Density of disk algebra functions in de
  {B}ranges--{R}ovnyak spaces}, Comptes Rendus Mathematique, 355 (2017),
  pp.~871--875.

\bibitem{jfabackshift}
{\sc A.~Aleman and B.~Malman}, {\em Hilbert spaces of analytic functions with a
  contractive backward shift}, Journal of Functional Analysis, 277 (2019),
  pp.~157--199.

\bibitem{conway1991theory}
{\sc J.~B. Conway}, {\em The theory of subnormal operators}, no.~36, American
  Mathematical Soc., 1991.

\bibitem{debranges2015square}
{\sc L.~de~Branges and J.~Rovnyak}, {\em {Square summable power series}},
  {Athena Series, Selected Topics in Mathematics. New York-Chicago-San
  Francisco-Toronto-London: Holt, Rinehart and Winston. viii, 104 p. (1966).},
  1966.

\bibitem{hbspaces1fricainmashreghi}
{\sc E.~Fricain and J.~Mashreghi}, {\em The theory of {$H$}({$b$}) spaces.
  {V}ol. 1}, vol.~20 of New Mathematical Monographs, Cambridge University
  Press, Cambridge, 2016.

\bibitem{hbspaces2fricainmashreghi}
{\sc E.~Fricain and J.~Mashreghi}, {\em The theory of {$H$}({$b$}) spaces.
  {V}ol. 2}, vol.~21 of New Mathematical Monographs, Cambridge University
  Press, Cambridge, 2016.

\bibitem{garnett}
{\sc J.~Garnett}, {\em Bounded analytic functions}, vol.~236, Springer Science
  \& Business Media, 2007.

\bibitem{hedenmalmbergmanspaces}
{\sc H.~Hedenmalm, B.~Korenblum, and K.~Zhu}, {\em Theory of Bergman spaces},
  vol.~199 of Graduate Texts in Mathematics, Springer-Verlag, New York, 2000.

\bibitem{helsonbook}
{\sc H.~Helson}, {\em Lectures on invariant subspaces}, Academic Press, New
  York-London, 1964.

\bibitem{khrushchev1978problem}
{\sc S.~V. Khrushchev}, {\em The problem of simultaneous approximation and of
  removal of the singularities of {C}auchy type integrals}, Trudy
  Matematicheskogo Instituta imeni VA Steklova, 130 (1978), pp.~124--195.

\bibitem{kriete1043splitting}
{\sc T.~L. Kriete}, {\em Splitting and boundary behavior in certain
  ${H}^2(\mu)$ spaces}, Linear and Complex Analysis Problem Book,(V. Havin et.
  al. Eds.), Lecture Notes in Math, 1043, pp.~439--446.

\bibitem{kriete1990mean}
{\sc T.~L. Kriete and B.~D. MacCluer}, {\em Mean-square approximation by
  polynomials on the unit disk}, Transactions of the American Mathematical
  Society, 322 (1990), pp.~1--34.

\bibitem{ptmuinnner}
{\sc A.~Limani and B.~Malman}, {\em Inner functions, invariant subspaces and
  cyclicity in ${\Po}^t(\mu)$-spaces}, preprint,  (2021).

\bibitem{smoothdensektheta}
{\sc A.~Limani and B.~Malman}, {\em On model spaces and density of functions
  smooth on the boundary}, 2021.

\bibitem{sarasonbook}
{\sc D.~Sarason}, {\em Sub-{H}ardy {H}ilbert spaces in the unit disk}, vol.~10
  of University of Arkansas Lecture Notes in the Mathematical Sciences, John
  Wiley \& Sons, Inc., New York, 1994.

\end{thebibliography}

\Addresses

\end{document}